\documentclass{article}
\usepackage{amssymb}
\usepackage{verbatim}
\usepackage{array}
\usepackage{latexsym}
\usepackage{enumerate}
\usepackage{amsmath}
\usepackage{amsfonts}
\usepackage{amsthm}
\usepackage{color}
\usepackage[english]{babel}
\usepackage{graphicx}
\usepackage[all]{xy}
\usepackage{boxedminipage}

\newtheorem{theorem}{Theorem}[section]

\newtheorem{lemma}[theorem]{Lemma}

\newtheorem{remark}[theorem]{Remark}

\def\cC{\mathcal C}

\def\cH{\mathcal H}

\newcommand{\fqs}{\mathbb {F}_{q^6}}

\def\Aut{\mbox{\rm Aut}}

\def\PG{{\rm{PG}}}
\def\ord{\mbox{\rm ord}}

\def\Aut{\mbox{\rm Aut}}

\newcommand{\PSL}{\mbox{\rm PSL}}
\newcommand{\PGL}{\mbox{\rm PGL}}

\newcommand{\PSU}{\mbox{\rm PSU}}
\newcommand{\PGU}{\mbox{\rm PGU}}

\newcommand{\aut}{\mbox{\rm Aut}}

\title{An $\mathbb{F}_{p^2}$-maximal Wiman's sextic and its automorphisms}
\date{}
\author{Massimo Giulietti, Motoko Kawakita, Stefano Lia and Maria Montanucci}

\begin{document}
\maketitle



\begin{abstract}
In 1895 Wiman introduced a Riemann surface $\mathcal{W}$  of genus $6$ over the complex field $\mathbb{C}$ defined by the homogeneous equation 
$\mathcal{W}:X^6+Y^6+Z^6+(X^2+Y^2+Z^2)(X^4+Y^4+Z^4)-12X^2 Y^2 Z^2=0$, and showed that  its full automorphism group is isomorphic to the symmetric group $S_5$. 
In \cite{kawakita} the curve $\mathcal{W}$ was studied as a curve defined over a finite field $\mathbb{F}_{p^2}$ where $p$ is a prime, and necessary and sufficient conditions for its maximality over $\mathbb{F}_{p^2}$ were obtained.
In this paper we first show that the result of Wiman concerning the automorphism group of $\mathcal{W}$ holds also over an algebraically closed field $\mathbb{K}$ of positive characteristic $p$, provided that $p \geq 7$. For $p=2,3$ the polynomial $X^6+Y^6+Z^6+(X^2+Y^2+Z^2)(X^4+Y^4+Z^4)-12X^2 Y^2 Z^2$ is not irreducible over $\mathbb{K}$, while for $p=5$ the curve $\mathcal{W}$ is rational and $Aut(\mathcal{W}) \cong \PGL(2,\mathbb{K})$. We also show that the $\mathbb{F}_{19^2}$-maximal Wiman's sextic $\mathcal{W}$ is not Galois covered by the Hermitian curve $\cH_{19}$ over $\mathbb{F}_{19^2}$.
\end{abstract}

\thanks{}

\section{Introduction} \label{SecIntro}
 
For $q$ a prime power, let $\mathbb{F}_{q^2}$ be the finite field with $q^2$ elements and $\cC$ be a projective, absolutely irreducible, non-singular algebraic curve of genus $g$ defined over $\mathbb{F}_{q^2}$. 
The curve $\cC$ is called $\mathbb{F}_{q^2}$-maximal if the number $|\cC(\mathbb{F}_{q^2})|$ of its $\mathbb{F}_{q^2}$-rational points attains the Hasse-Weil upper bound
$$q^2+1+2gq.$$
%Maximal curves have been investigated for their applications in Coding Theory. 
Surveys on maximal curves are found in  \cite{FT,G,G2,GS,vdG,vdG2} and \cite[Chapter 10]{HKT}.
By a result commonly referred as the Kleiman-Serre covering result, see \cite{KS} and \cite[Proposition 6]{L}, a curve   $\cC$ defined over $\mathbb{F}_{q^2}$ which is $\mathbb{F}_{q^2}$-covered  by an $\mathbb{F}_{q^2}$-maximal curve is $\mathbb{F}_{q^2}$-maximal as well. In particular, $\mathbb{F}_{q^2}$-maximal curves can be obtained as Galois $\mathbb{F}_{q^2}$-subcovers of an $\mathbb{F}_{q^2}$-maximal curve $\cC$, that is, as quotient curves $\cC/G$ where $G$ is a finite automorphism group of $\cC$ defined over $\mathbb F_{q^2}$. Most of the known $\mathbb F_{q^2}$-maximal curves are Galois covered by the Hermitian curve $\cH_q: X^q+X=Y^{q+1}$; see e.g. \cite{GSX,CKT2,GHKT2} and the references therein.

The first example of a maximal curve which is not Galois covered by the Hermitian curve was discovered by Garcia and Stichtenoth \cite{GS3}. This curve is $\mathbb{F}_{3^6}$-maximal and it is not Galois covered by $\cH_{27}$. It is a special case of the $\mathbb{F}_{q^6}$-maximal GS curve, which was later shown not to be Galois covered by $\cH_{q^3}$ for any $q>3$, \cite{GMZ,Mak}. Giulietti and Korchm\'aros \cite{GK} provided an $\mathbb{F}_{q^6}$-maximal curve, nowadays referred to as the GK curve, which is not covered by the Hermitian curve $\cH_{q^3}$ for any $q>2$. Two generalizations of the GK curve were introduced by Garcia, G\"uneri and Stichtenoth \cite{GGS} and by Beelen and Montanucci in \cite{BM}. Both these two generalizations are $\mathbb{F}_{q^{2n}}$-maximal curves, for any $q$ and odd $n \geq 3$. Also, they are not Galois covered by the Hermitian curve $\cH_{q^n}$ for $q>2$ and $n \geq 5$, see \cite{DM,BM}; the Garcia-G\"uneri-Stichtenoth's generalization is also not Galois covered by $\cH_{2^n}$ for $q=2$, see \cite{GMZ}.
The existence of infinite families of $\mathbb F_{p^2}$-maximal curves that are not Galois covered by $\mathcal H_p$ is an interesting open problem.

In 1895 Wiman \cite{wiman} introduced a Riemann surface $\mathcal{W}$ over the complex field $\mathbb{C}$ defined by the homogeneous equation of degree $6$
$$\mathcal{W}: X^6+Y^6+Z^6+(X^2+Y^2+Z^2)(X^4+Y^4+Z^4)-12X^2 Y^2 Z^2=0,$$
whose full automorphism group is isomorphic to the symmetric group $S_5$. The Jacobian of $\mathcal{W}$ decomposes completely as the product of the Jacobian of an elliptic curve $\epsilon$ six times. This fact, together with the nice projective model of $\mathcal{W}$, stimulated the investigation of the reduction $\rm{mod}$ $p$ of the curve $\mathcal{W}$ and its properties over finite fields of characteristic $p$.
Indeed Kawakita \cite{kawakita} used the complete decomposition of the Jacobian of $\mathcal{W}$ to apply Kani-Rosen Theorem \cite{KR}, and  obtained necessary and sufficient conditions for the maximality of $\mathcal{W}$ over $\mathbb{F}_{p^2}$. In particular, $\mathcal W$ is $\mathbb F_{p^2}$-maximal for infinite primes $p$; see also Section \ref{SecPrel}.  

In this paper the result of Wiman concerning the structure of the full automorphism group of $\mathcal{W}$ is extended to any algebraically closed field of characteristic $p \geq 7$. For $p=2,3$ the homogeneous polynomial defining $\mathcal{W}$ is not irreducible, while for $p=5$ the curve $\mathcal{W}$ is rational. 

We also show that the $\mathbb{F}_{19^2}$-maximal curve $\mathcal{W}$ is not Galois covered by the Hermitian curve $\cH_{19}$ over $\mathbb{F}_{19^2}$. This makes the Wiman sextics a natural candidate to provide the first example of 
an infinite family of $\mathbb F_{p^2}$-maximal curves that are not covered by the Hermitian curve.

The paper is organized as follows. Section \ref{SecPrel} provided a collection of preliminary results on the Hermitian curve and its automorphisms, on the Wiman's sextic $\mathcal{W}$ and on automorphism groups of algebraic curves in characteristic $p>0$. 
In Section \ref{secautw1} the full automorphism group of $\mathcal{W}$ is computed over algebraically closed field of characteristic $p \geq 7$, extending the result of Wiman \cite{wiman} to the positive characteristic case. 
Finally, in Section \ref{nonga}, the $\mathbb{F}_{19^2}$-maximal Wiman's sextic $\mathcal{W}$ is shown not to be Galois covered by the Hermitian curve  $\cH_{19}$ over $\mathbb{F}_{19^2}$.

\section{Preliminary results}\label{SecPrel}

\subsection{The Wiman's sextic $\mathcal{W}$} 

In 1895 Wiman introduced in \cite{wiman} a Riemann surface $\mathcal{W}$ over the complex field $\mathbb{C}$ defined by the homogeneous equation of degree $6$
\begin{equation} \label{eqWiman1}
\mathcal{W}: X^6+Y^6+Z^6+(X^2+Y^2+Z^2)(X^4+Y^4+Z^4)-12X^2 Y^2 Z^2=0.
\end{equation}
The irreducible curve $\mathcal{W}$ has genus $6$ and $4$ ordinary double points, namely $[1:1:1]$, $[1:-1:1]$, $[-1:1:1]$, $[-1:-1:1]$. 

In \cite{kawakita} the curve $\mathcal{W}$ is studied as a curve defined over a finite field $\mathbb{F}_{p^2}$, where $p$ is a prime. Primes $p$  for which $\mathcal{W}$ is $\mathbb F_p$-maximal are characterized by applying the Kani-Rosen Theorem \cite{KR}, taking into account that the Jacobian $J_{\mathcal{W}}$ of $\mathcal{W}$ decomposes completely as $J_{\mathcal{W}} \sim \epsilon^6$ where $\epsilon$ is the elliptic curve defined by the homogeneous equation $\epsilon: Y^2Z-X(5X^2-95XZ+2^9Z^3)$.
It turns out that $\mathcal{W}$ is $\mathbb{F}_{p^2}$-maximal if and only if the characteristic $p$ satisfies the equation:
\begin{equation}\label{kawakita}
 \sum_{i=0}^{ \lfloor m/2 \rfloor }  \frac{m!}{(i!)^2(m-2i)!} 2^{9i}5^{-m-i}(-19)^{m-2i} \equiv 0   \pmod p. 
\end{equation}
Explicit values of $p$ for which Equation \eqref{kawakita} is satisfied are also found in \cite{kawakita}:

 $$p=19,29,79,199,269,359,439,499,509,599,919,1279.$$

Wiman showed that the automorphism group $Aut(\mathcal{W})$ of $\mathcal{W}$ over $\mathbb{C}$ is isomorphic to ${\rm S}_5$, the symmetric group on $5$ letters. 

In Section \ref{secautw1} we show that the same result holds true over an algebraically closed field of positive characteristic $p\ge 7$. The reason why $p$ has to be at least $7$ is that for $p=2,3$ the polynomial \eqref{eqWiman1} is reducible, while for $p=5$  the curve $\mathcal{W}$ is rational and hence its full automorphism group is the projective general linear group over the algebraic closure of $\mathbb F_5$.

\subsection{The Hermitian curve and its automorphism group}
Throughout this section, $q=p^n$, where $p$ is a prime number, $n$ is a positive integer and $\mathbb{K}$ is the algebraic closure of the finite field with $q$ elements $\mathbb{F}_q$. The Deligne-Lusztig curves defined over  $\mathbb{F}_q$ were originally introduced in \cite{DL}. Other than the projective line, there are three families of Deligne-Lusztig curves, named Hermitian curves, Suzuki curves and Ree curves. The Hermitian curve $\mathcal H_q$ arises from the algebraic group $^2A_2(q)={\rm PGU}(3,q)$ of order $(q^3+1)q^3(q^2-1)$. It has genus $q(q-1)/2$ and is $\mathbb F_{q^2}$-maximal. This curve is isomorphic to the curves listed below:
\begin{equation} \label{M1}
X^{q+1}+Y^{q+1}+Z^{q+1}=0;
\end{equation}
\begin{equation} \label{M2}
X^{q}Z+XZ^{q}-Y^{q+1}=0;
\end{equation}
\begin{equation} \label{M3}
XY^{q}-YX^{q}+\omega Z^{q+1}=0,
\end{equation}
where $\omega$ is a fixed element of $\mathbb{K}$ such that $\omega^{q-1}=-1$;
\begin{equation} \label{M4}
XY^{q}+YZ^{q}+\omega ZX^{q}=0.
\end{equation}
Each of the models (\ref{M1}),(\ref{M2}) and (\ref{M3}) is $\mathbb{F}_{q^2}$-isomorphic to $\cH_q$, while the model (\ref{M4}) is $\mathbb{F}_{q^6}$-isomorphic to $\cH_q$, since for a suitable element $a \in \mathbb{F}_{q^6}$, the projective map 
$$k: \mathbb{P}^2(\mathbb{K}) \rightarrow \mathbb{P}^2(\mathbb{K}), \ (X:Y:Z) \mapsto (aX+Y+a^{q^2+1}Z: a^{q^2+1}X+aY+Z:x+a^{q^2+1}Y+aZ),$$
changes  (\ref{M1}) into (\ref{M4}), see \cite[Proposition 4.6]{CKT1}.

The automorphism group $\aut(\cH_q)$ is isomorphic to the projective unitary group $\PGU(3,q)$, and it acts on the set $\cH_q(\mathbb{F}_{q^2})$ of all $\mathbb{F}_{q^2}$-rational points of $\cH_q$ as $\PGU(3,q)$ in its usual $2$-transitive permutation representation.
The combinatorial properties of $\cH_q(\mathbb{F}_{q^2})$ can be found in \cite{HP}. The size of $\cH_q(\mathbb{F}_{q^2})$ is equal to $q^3+1$, and a line of ${\rm PG}(2,q^2)$ has either $1$   or $q+1$ common points with $\cH_q(\mathbb{F}_{q^2})$. 
In the former case the line is said to be a $1$-secant or  tangent, in the former case a $(q+1)$-secant or chord.
A unitary polarity is associated with $\cH_q(\mathbb{F}_{q^2})$; its isotropic points are those in $\cH_q(\mathbb{F}_{q^2})$, and its isotropic lines are the $1$-secants of $\cH_q(\mathbb{F}_{q^2})$, that is, the tangents to $\cH_q$ at the points of $\cH_q(\mathbb{F}_{q^2})$.

A useful tool in our investigation is the classification of maximal subgroups of the projective special subgroup $\PSU(3,q)$ of $\PGU(3,q)$, going back to Mitchell and Hartley; see \cite{M}, \cite{H}, \cite{HO}.

\begin{theorem} \label{Mit} Let $d={\rm gcd}(3,q+1)$. Up to conjugacy, the subgroups below give a complete list of maximal subgroups of $\PSU(3,q)$.

\begin{itemize}
\item[(i)] the stabilizer of an $\mathbb{F}_{q^2}$-rational point of $\cH_q$. It has order $q^3(q^2-1)/d$;
\item[(ii)] the stabilizer of an $\mathbb{F}_{q^2}$-rational point off $\cH_q$ $($equivalently the stabilizer of a chord of $\cH_q(\mathbb{F}_{q^2}))$. It has order $q(q-1)(q+1)^2/d$;
\item[(iii)] the stabilizer of a self-polar triangle with respect to the unitary polarity associated to $\cH_q(\mathbb{F}_{q^2})$. It has order $6(q+1)^2/d$;
\item[(iv)] the normalizer of a (cyclic) Singer subgroup. It has order $3(q^2-q+1)/d$ and preserves a triangle
in $\PG(2,q^6)\setminus\PG(2,q^2)$ left invariant by the Frobenius collineation $\Phi_{q^2}:(X,Y,T)\mapsto (X^{q^2},Y^{q^2},T^{q^2})$ of $\PG(2,\mathbb{K})$;

{\rm for $p>2$:}
\item[(v)] ${\rm PGL}(2,q)$ preserving a conic;
\item[(vi)] $\PSU(3,p^m)$ with $m\mid n$ and $n/m$ odd;
\item[(vii)] subgroups containing $\PSU(3,p^m)$ as a normal subgroup of index $3$, when $m\mid n$, $n/m$ is odd, and $3$ divides both $n/m$ and $q+1$;
\item[(viii)] the Hessian groups of order $216$ when $9\mid(q+1)$, and of order $72$ and $36$ when $3\mid(q+1)$;
\item[(ix)] ${\rm \PSL(2,7)}$ when $p=7$ or $-7$ is not a square in $\mathbb{F}_q$;
\item[(x)] the alternating group ${\rm A}_6$ when either $p=3$ and $n$ is even, or $5$ is a square in $\mathbb{F}_q$ but $\mathbb{F}_q$ contains no cube root of unity;
\item[(xi)] the symmetric group ${\rm S}_6$ when $p=5$ and $n$ is odd;
\item[(xii)] the alternating group ${\rm A}_7$ when $p=5$ and $n$ is odd;

{\rm for $p=2$:}
\item[(xiii)] $\PSU(3,2^m)$ with $m\mid n$ and $n/m$ an odd prime;
\item[(xiv)] subgroups containing $\PSU(3,2^m)$ as a normal subgroup of index $3$, when $n=3m$ with $m$ odd;
\item[(xv)] a group of order $36$ when $n=1$.
\end{itemize}
\end{theorem}

In the following, a subgroup $G \leq \PGU(3,q)$ is said to be \textit{tame} if its order is coprime to $p$ and \textit{non-tame} otherwise.

In our investigation it is also useful to know how an element of $\PGU(3,q)$ of a given order acts on $\PG(2,\mathbb{K})$, and in particular on $\cH_q(\mathbb{F}_{q^2})$. This is stated in Lemma \ref{classificazione} with the usual terminology about collineations of projective planes; see e.g. \cite{HP}. In particular, a linear collineation $\sigma$ of $\PG(2,\mathbb{K})$ is a $(P,\ell)$-\emph{perspectivity}, if $\sigma$ preserves  each line through the point $P$ (the \emph{center} of $\sigma$), and fixes each point on the line $\ell$ (the \emph{axis} of $\sigma$). A $(P,\ell)$-perspectivity is either an \emph{elation} or a \emph{homology} according as $P\in \ell$ or $P\notin\ell$. A $(P,\ell)$-perspectivity is in  $\PGL(3,q^2)$ if and only if its center and its axis are in $\PG(2,\mathbb{F}_{q^2})$.
\begin{lemma}{\rm{(}\cite[Lemma 2.3]{MZRS}\rm{)}}\label{classificazione}
For a nontrivial element $\sigma\in \PGU(3,q)$, one of the following cases holds.
\begin{itemize}
\item[(A)] ${\rm ord}(\sigma)\mid(q+1)$ and $\sigma$ is a homology whose center $P$ is a point off $\cH_q$ and whose axis $\ell$ is a chord of $\cH_q(\mathbb{F}_{q^2})$ such that $(P,\ell)$ is a pole-polar pair with respect to the unitary polarity associated to $\cH_q(\mathbb{F}_{q^2})$.
\item[(B)] ${\rm ord}(\sigma)$ is coprime to $p$ and $\sigma$ fixes the vertices $P_1,P_2,P_3$ of a non-degenerate triangle $T$.
\begin{itemize}
\item[(B1)] The points $P_1,P_2,P_3$ are $\fqs$-rational, $P_1,P_2,P_3\notin\cH_q$ and the triangle $T$ is self-polar with respect to the unitary polarity associated to $\cH_q(\mathbb{F}_{q^2})$. Also, $\ord(\sigma)\mid(q+1)$.
\item[(B2)] The points $P_1,P_2,P_3$ are $\fqs$-rational, $P_1\notin\cH_q$, $P_2,P_3\in\cH_q$. %and the polar lines of $P_1,P_2,P_3$ are $P_2P_3$, $P_1P_2$, $P_1P_3$, respectively.
     Also, $\ord(\sigma)\mid(q^2-1)$ and $\ord(\sigma)\nmid(q+1)$.
\item[(B3)] The points $P_1,P_2,P_3$ have coordinates in $\mathbb{F}_{q^6}\setminus\mathbb{F}_{q^2}$, $P_1,P_2,P_3\in\cH_q$. %and their polar lines are $P_1P_2$, $P_2P_3$, $P_3P_1$, respectively.
    Also, $\ord(\sigma)\mid (q^2-q+1)$.
\end{itemize}
\item[(C)] ${\rm ord}(\sigma)=p$ and $\sigma$ is an elation whose center $P$ is a point of $\cH_q$ and whose axis $\ell$ is a tangent of $\cH_q(\mathbb{F}_{q^2})$; here $(P,\ell)$ is a pole-polar pair with respect to the unitary polarity associated to $\cH_q(\mathbb{F}_{q^2})$.
\item[(D)] ${\rm ord}(\sigma)=p$ with $p\ne2$, or ${\rm ord}(\sigma)=4$ and $p=2$. In this case $\sigma$ fixes an $\fqs$-rational point $P$, with $P \in \cH_q$, and a line $\ell$ which is a tangent  of $\cH_q(\mathbb{F}_{q^2})$; here $(P,\ell)$ is a pole-polar pair with respect to the unitary polarity associated to $\cH_q(\mathbb{F}_{q^2})$.
\item[(E)] $p\mid{\rm ord}(\sigma)$, $p^2\nmid{\rm ord}(\sigma)$, and ${\rm ord}(\sigma)\ne p$. In this case $\sigma$ fixes two $\fqs$-rational points $P,Q$, %the line $PQ$ linewise, and another line $\ell$ through $P$ linewise.
     with $P\in\cH_q$, $Q\notin\cH_q$. %$PQ$ is the polar line of $P$, and $\ell$ is the polar line of $Q$.
\end{itemize}
\end{lemma}

Throughout the paper, a nontrivial element of $\PGU(3,q)$ is said to be of type (A), (B), (B1), (B2), (B3), (C), (D), or (E), as given in Lemma \ref{classificazione}.

Every subgroup $G$ of $\PGU(3,q)$ produces a quotient curve $\cH_q/G$, and the cover $\cH_q\rightarrow\cH_q/G$ is a Galois cover defined over $\mathbb{F}_{q^2}$  where the degree of the different divisor $\Delta$ is given by the Riemann-Hurwitz formula \cite[Theorem 3.4.13]{Sti},
\begin{equation} \label{RHformula}
\Delta=(2g(\cH_q)-2)-|G|(2g(\cH_q/G)-2).
\end{equation}

On the other hand, $\Delta=\sum_{\sigma\in G\setminus\{id\}}i(\sigma)$, where $i(\sigma)\geq 0$ is given by the Hilbert's different formula \cite[Thm. 3.8.7]{Sti}, namely
\begin{equation}\label{contributo}
\textstyle{i(\sigma)=\sum_{P\in\cH_q(\bar{\mathbb{F}}_q)}v_P(\sigma(t)-t),}
\end{equation}
where $t$ is a local parameter at $P$.

By analyzing the geometric properties of the elements $\sigma \in \PGU(3,q)$, it turns out that there are only a few possibilities for $i(\sigma)$.
This is obtained as a corollary of Lemma \ref{classificazione} and stated in the following proposition, see \cite{MZRS}.

\begin{theorem}{\rm{(}\cite[Theorem 2.7]{MZRS}\rm{)}}\label{caratteri}
For a nontrivial element $\sigma\in \PGU(3,q)$ one of the following cases occurs.
\begin{enumerate}
\item If $\ord(\sigma)=2$ and $2\mid(q+1)$, then $\sigma$ is of type {\rm(A)} and $i(\sigma)=q+1$.
\item If $\ord(\sigma)=3$, $3 \mid(q+1)$ and $\sigma$ is of type {\rm(B3)}, then $i(\sigma)=3$.
\item If $\ord(\sigma)\ne 2$, $\ord(\sigma)\mid(q+1)$ and $\sigma$ is of type {\rm(A)}, then $i(\sigma)=q+1$.
\item If $\ord(\sigma)\ne 2$, $\ord(\sigma)\mid(q+1)$ and $\sigma$ is of type {\rm(B1)}, then $i(\sigma)=0$.
\item If $\ord(\sigma)\mid(q^2-1)$ and $\ord(\sigma)\nmid(q+1)$, then $\sigma$ is of type {\rm(B2)} and $i(\sigma)=2$.
\item If $\ord(\sigma)\ne3$ and $\ord(\sigma)\mid(q^2-q+1)$, then $\sigma$ is of type {\rm(B3)} and $i(\sigma)=3$.
\item If $p=2$ and $\ord(\sigma)=4$, then $\sigma$ is of type {\rm(D)} and $i(\sigma)=2$.
\item If $\ord(\sigma)=p$, $p \ne2$ and $\sigma$ is of type {\rm(D)}, then $i(\sigma)=2$.
\item If $\ord(\sigma)=p$ and $\sigma$ is of type {\rm(C)}, then $i(\sigma)=q+2$.
\item If $\ord(\sigma)\ne p$, $p\mid\ord(\sigma)$ and $\ord(\sigma)\ne4$, then $\sigma$ is of type {\rm(E)} and $i(\sigma)=1$.
\end{enumerate}
\end{theorem}

\subsection{Automorphism groups of algebraic curves}
This section provides a collection of preliminary results on automorphism groups of algebraic curves that will be used in the following sections for the determination of the full automorphism group of the Wiman's sextic $\mathcal{W}$.

\begin{theorem}{\rm (\cite[Theorem 11.78]{HKT})}\label{th11.78}
Let $\cC$ be an irreducible algebraic curve of genus $g \geq 1$ defined over a field of characteristic $p$ and let $G$ be an automorphism group of $\cC$. Let  $G_P$ be the stabilizer of a place $P$ of $\cC$
and $G^{(i)}_P$ be the $i$-th ramification group of $G$ at $P$. Then
\begin{equation}\label{11.78}
|G_P| \leq \frac{4p}{p-1}g^2.
\end{equation}
Also, if $\cC_i$ denotes the quotient curve $\cC/G^{(i)}_P$,  then one of the following cases occurs: 
\begin{itemize}
\item[(i)] $\cC_1$ is not rational, and $|G^{(1)}_P|\leq g$; 

\item[(ii)] $\cC_1$ is rational, $G^{(1)}_P$ has a short orbit other than $\{ P\}$, and 
$$|G^{(1)}_P|\leq \frac{p}{p-1}g;$$

\item[(iii)] $\cC_1$ and $\cC_2$ are rational, $\{ P \}$ is the unique short orbit of $G^{(1)}_P$, and
$$|G^{(1)}_P|\leq\frac{4p}{(p-1)^2}g^2.$$

\end{itemize}
\end{theorem}
\begin{theorem}{\rm (\cite[Theorem 11.56]{HKT})}\label{tame}
Let $\cC$ be an irreducible curve of genus $g \geq 2$ over a field $\mathbb{K}$. If $char(\mathbb{K})=0$ or $char(\mathbb{K})=p >0$ with $(p,|Aut(\cC)|)=1$ then 
\begin{equation}\label{hb}
|Aut(\cC)| \leq 84(g-1).
\end{equation}
\end{theorem}

The previous result is known in the literature as \textit{Classical Hurwitz bound}. The following theorem describes the short orbits structure of an automorphism group $G \leq Aut(\cC)$ of an algebraic curve $\cC$ of genus $g \geq 2$ for which the Classical Hurwitz bound does not hold. 

\begin{theorem}{\rm (\cite[Theorem 11.126 and Theorem 11.56]{HKT})} \label{thmHurwitz}
Let $\cC$ be an irreducible curve of genus $g \geq 2$ and let $G \leq Aut(\cC)$ with $|G| > 84(g-1).$ Then the quotient curve $\cC / G$ is rational and $G$ has at most three short orbits as follows.
\begin{enumerate}
\item Exactly three short orbits, two tame and one non-tame. Each point in the tame short orbits has stabilizer in $G$ of order $2$; 
\item exactly two short orbits, both non-tame;
\item only one short orbit which is non-tame;
\item exactly two short orbits, one tame and one non-tame. In this case $|G|< 8g^3$, with the following exceptions:
\begin{itemize}
\item $p=2$ and $\cC$ is isomorphic to the hyperelliptic curve $y^2+y=x^{2^k+1}$ with genus $2^{k-1}$ ;
\item $p>2$ and $\cC$ is isomorphic to the Roquette curve $y^2=x^q-x$ with genus $(q-1)/2$ ;
\item $p\geq 2$ and $\cC$ is isomorphic to the Hermitian curve $y^{q+1}=x^q+x$ with genus $(q^2-q)/2$ ;
\item $p=2$ and $\cC$ is isomorphic to the Suzuki curve $y^q+y=x^{q_0}(x^q+x)$ with genus $q_0(q-1)$ .
\end{itemize}
\end{enumerate}
\end{theorem}

%\begin{lemma}{\rm \cite[Lemma 4.1]{mbt}} \label{mbt}
%Let $p \geq 7$ and let $\cC$ denote an irreducible $\mathbb{F}_{p^2}$-maximal curve of genus $g \geq 2$.
 %If $$40(g-1) < |\Aut(\cC)| < 84(g-1),$$ and $Aut(\cC)$ is tame then 
%\begin{equation}\label{mbt eq}
%|Aut(\cC)|=48(g-1).
%\end{equation}
%In particular, $\Aut(\cC)$ has exactly three short orbits of lengths $|\Aut(\cC)|/2$, $|\Aut(\cC)|/3$ and $|%\Aut(\cC)|/8$ respectively.
%\end{lemma}

The following lemma considers the short orbits structure of a large tame automorphism group of a curve $\cC$ of genus $g \geq 2$.

\begin{lemma} \label{mbt}
Let $p$ be a prime. Let $\cC$ be an irreducible algebraic curve of genus $g \geq 2$ defined over a field of characteristic $p$ such that $40(g-1)<|Aut(\cC)| \leq 84(g-1)$. Assume also that $Aut(\cC)$ is tame. 

Then $\aut(\cC)$ has exactly $3$ tame short orbits $O_i$ for $i=1,2,3$, $\cC/ \aut(\cC)$ is rational, and one of the following cases occurs.
\begin{enumerate}
\item $|O_1|=|\aut(\cC)|/2$, $|O_2|=|\aut(\cC)|/3$, $|O_3|=|\aut(\cC)|/7$ and $p \geq 11$. In this cases $\aut(\cC)$ has order $84(g-1)$; 
\item $|O_1|=|\aut(\cC)|/2$, $|O_2|=|\aut(\cC)|/3$ and $|O_3|=|\aut(\cC)|/8$. In this cases $|\aut(\cC)|=48(g-1)$.
\end{enumerate}
\end{lemma}

\begin{proof}
From the Riemann-Hurwitz formula
\begin{equation} \label{eqdiff}
2g-2=|\aut(\cC)|(2g^\prime -2 +d^\prime),
\end{equation}
where $d_P^\prime=d_p/|\aut(\cC)_P|$ and $d^\prime = \sum_{P} d_P^\prime$. Here the summation is only over a set of representatives of places in $\cC$, exactly one from each short orbit of $Aut(\cC)$. 

So, it is necessary to investigate the possibilities for $|\aut(\cC)|$ according to the number $r$ of short orbits of $Aut(\cC)$ on $\cC$. From the Hilbert different formula, see \cite[Theorem 11.70]{HKT}, $d_P \geq e_P-1$, with equality holding if and only if $e_P$ is prime to $p$. Therefore, if $d_P>0$, then $d_P^\prime \geq 1/2$. More precisely, since $Aut(\cC)$ is tame, $d_P^\prime=(e_P-1)/e_P$. 
Hence, if $d>0$, then $d^\prime \geq 1/2$. 
If $g^\prime \geq 2$ then $|Aut(\cC)| \leq g-1$, a contradiction. For $g^\prime=1$, it follows that $d^\prime>0$ since $g \geq 2$, and hence $|Aut(\cC)| \leq 4(g-1)$, a contradiction. Thus $g^\prime=0$. Then
$$2g-2=|Aut(\cC)|(d^\prime -2).$$

In particular, $d^\prime >2$. Therefore $G$ has some, say $r \geq 1$, short orbits on $\cC$. 
Take representatives $Q_1,\ldots,Q_r$, from each short orbit, and let $d_i^\prime=d_{Q_i}^\prime$ for $i=1,\ldots,r$. After a change of indices, it may be assumed that $d_i^\prime \leq d_j^\prime$ for $i \leq j$. 

\begin{itemize}
\item When $r \geq 5$, then $d^\prime \geq 5/2$, and hence $|Aut(\cC)| \leq 4(g-1)$.
\item When $r=4$ then $d^\prime>2$ and $d_i^\prime >1/2$ for at least one place $P$. As $d_i^\prime>1/2$ implies $d_i^\prime \geq 2/3$, so $d^\prime-2 \geq 1/6$, whence $|Aut(\cC)| \leq 12(g-1)$.
\item When $r=3$ then again use $d^\prime>2$. If $d_1^\prime=2/3$ then $d_3^\prime \geq 3/4$ and hence $|Aut(\cC)| \leq 24(g-1)$. If $d_1^\prime=1/2$ and $d_2^\prime \geq 3/4$ then $|Aut(\cC)| \leq 40(g-1)$. 
Thus, assume that $d_1^\prime=1/2$ and $d_2^\prime=2/3$. From \eqref{eqdiff}, 
$$2(g-1)=|Aut(\cC)|(d^\prime-2)=|Aut(\cC)|\bigg(\frac{1}{2}+\frac{2}{3}+d_3^\prime-2\bigg)>40(g-1)\bigg(\frac{1}{2}+\frac{2}{3}+d_3^\prime-2\bigg),$$
implying that $d_3^\prime < 53/60$. Also $d^\prime=1/2+2/3+d_3^\prime-2>0$ and hence $d_3^\prime>5/6$, giving $d_3^\prime \geq 6/7$. 
Thus, we get that $6/7 \leq d_3^\prime < 53/60$, and hence either $d_3^\prime=6/7$ or $d_3^\prime=7/8$. Now one of the cases 1 and 2 occurs. 
\item When $r=2$ then $d^\prime=d_1^\prime+d_2^\prime>2$. This can only occur when either $d_1^\prime$ or $d_2^\prime$ or both are greater than $1$. Hence, one of cases 2 and 4 of Theorem \ref{thmHurwitz} occurs.
This is not possible as we are assuming that $Aut(\cC)$ is tame.

\item When $r=1$ then $d^\prime=d_1^\prime>2$, and case 3 of Theorem \ref{thmHurwitz} occur.
Again, this is not possible as we are assuming that $Aut(\cC)$ is tame.
\end{itemize}
\end{proof}

A careful analysis of the automorphism group of algebraic curves $\cC$ of even genus $g \geq 2$ can be found in \cite{gkaut}. The following result provides some restrictions to the structure of an automorphism group of $\cC$ admitting a minimal normal subgroup of order $4$. We recall that a \textit{minimal normal subgroup} $N$ of a group $G$ is a normal subgroup of $G$ such that the only normal subgroup of $G$ properly contained in $N$ is the trivial subgroup.

\begin{lemma}{\rm (\cite[Lemma 6.6]{gkaut})}\label{gkaut}
 Let $\cC$ be an irreducible curve of even genus $g \geq 2$ defined over a finite field of odd characteristic. If the automorphism group $Aut(\cC)$ of $\cC$ has a minimal normal subgroup $N$ of order $4$, then either for $Aut(\cC)$, or for a normal subgroup of $G$ of index $3$, the following condition is satisfied:
\begin{itemize}

\item $G = O(G) \rtimes S_2$ where $S_2$ is the direct product of a cyclic group by a group of order $2$.
\end{itemize}
\end{lemma}

Let $\cC$ be an algebraic curve of genus $g \geq 2$ defined over an algebraically closed field $\mathbb{K}$ of positive characteristic $p$. If $|Aut(\cC)|$ is divisible by $p$ then bounds for the order of a Sylow $p$-subgroup of $Aut(\cC)$ can be found in \cite{nak}. 

In the following Theorem bounds for the order of a Sylow $p$-subgroup of $Aut(\cC)$ are written with respect to an important birational invariant of $\cC$, namely its \textit{$p$-rank} $\gamma$. It is defined to be the rank of the (elementary abelian) group of the $p$-torsion points in the Jacobian variety of $\cC$; moreover, $\gamma \leq g$ and when the equality holds then $\cC$ is called an \textit{ordinary} (or \textit{general}) curve; see \cite[Section 6.7]{HKT}.

\begin{theorem}{\rm (\cite[Theorem 1 (i)]{nak})} \label{nak}
Let $\cC$ be a curve of genus $g \geq 2$ and $p$-rank $\gamma$ defined over an algebraically closed field of positive characteristic $p$. Let $H$ be a Sylow $p$-subgroup of $Aut(\cC)$. If $\gamma \geq 2$ then
\begin{equation} \label{nakbound}
|H| \leq c_p (\gamma-1),
\end{equation}
where $c_p=p/(p-2)$ for $p \geq 3$ and $c_2=4$.
\end{theorem}

Curves $\cC$ together with a $p$-group $H$ of automorphisms such that the bound \eqref{nakbound} is attained are called \textit{Nakajima extremal curves}. Giulietti and Korchm\'aros in \cite{GKext} showed that the full automorphism group of Nakajima extremal curves has a precise structure.

\begin{theorem}{\rm (\cite[Theorem 1.3]{GKext})} \label{nakext}
Let $\cC$ be a Nakajima extremal curve, and $H$ be a Sylow $p$-subgroup of $Aut(\cC)$. Then either $H$ is a normal subgroup of $Aut(\cC)$ and $Aut(\cC)$ is the semidirect product of $H$ by a subgroup of a dihedral group of order $2(p-1)$, or $p=3$ and, for some subgroup $M$ of $H$ of index $3$, $M$ is a normal subgroup of $Aut(\cC)$ and $Aut(\cC)/M$ is isomorphic to a subgroup of ${\rm GL}(2,3)$.
\end{theorem}

\section{The automorphism group of $\mathcal{W}$} \label{secautw1}

Wiman proved that the automorphism group of $\mathcal{W}$ over the complex field $\mathbb C$ is the symmetric group ${\rm S}_5$. In this section we show that this holds true also in positive characteristics $p$ when $p \geq 7$. 
Let $$F(X,Y,Z)=X^6+Y^6+Z^6+(X^2+Y^2+Z^2)(X^4+Y^4+Z^4)-12X^2 Y^2 Z^2$$ be the defining polynomial of  $\mathcal{W}$.
First of all, in the following remark, we justify the hypothesis on $p$.

\begin{remark} \label{ppiccoli}
Let $\mathcal{W}$ be the Wiman's sextic defined as in Equation \eqref{eqWiman1} over a field of characteristic $p$. Then the following holds.
\begin{itemize}
\item If $p=2$ or $p=3$ then the polynomial $F(X,Y,Z) \in \mathbb{F}_{p}[X,Y,Z]$ is reducible.
\item If $p=5$ then $\mathcal{W}$ is rational. In particular from {\rm \cite[Theorem 11.14]{HKT}}, $Aut(\mathcal{W}) \cong \PGL(2,\mathbb{K})$, where $\mathbb{K}$ is the algebraic closure of $\mathbb{F}_p$.
\end{itemize}
\end{remark}

From now on we assume that $p \geq 7$.  

First of all, we note that the following three rational maps are automorphisms of $\mathcal{W}$ of order $2$:
\begin{equation} \label{tauphirho}
\phi: [X:Y:Z]\mapsto [X:-Y:Z], \quad
\tau: [X:Y:Z]\mapsto [-X:Y:Z], \quad
\rho: [X:Y:Z]\mapsto [Y:X:Z].
\end{equation}

Likewise by direct checking, it is easily seen that the following map provides an automorphism of order $3$ of $\mathcal{W}$:
\begin{equation} \label{gamma}
\gamma: [X:Y:Z] \mapsto [-Z:-X:Y].
\end{equation}
A more complicated automorphism of order $5$ of $\mathcal{W}$ is given by the following map:
\begin{equation} \label{alpha}
\alpha:[X:Y:Z]\mapsto [g_0(X,Y,Z):g_1(X,Y,Z):g_2(X,Y,Z)],
\end{equation}
where 
$$g_0(X,Y,Z)=-X^2+XY+XZ-Y^2-YZ+Z^2,$$
$$g_1(X,Y,Z)=-X^2+XY+XZ+Y^2-YZ-Z^2,$$
$$g_2(X,Y,Z)=X^2+XY+XZ-Y^2-YZ-Z^2.$$

Indeed the following computation shows that $\alpha$ is an automorphism of $\mathcal{W}$,

$$\alpha(F(X,Y,Z))=\alpha(X^6+Y^6+Z^6+(X^2+Y^2+Z^2)(X^4+Y^4+Z^4)-12X^2 Y^2 Z^2) =$$
$$ g_0(X,Y,Z)^6+g_1(X,Y,Z)^6+g_2(X,Y,Z)^6+$$
$$+(g_0(X,Y,Z)^2+g_1(X,Y,Z)^2+g_2(X,Y,Z)^2)(g_0(X,Y,Z)^4+g_1(X,Y,Z)^4+g_2(X,Y,Z)^4)+$$
$$-12g_0(X,Y,Z)^2 g_1(X,Y,Z)^2 g_2(X,Y,Z)^2=$$
$$2^6(Y+Z)^2(X-Z)^2(X-Y)^2(2X^6 + X^4Y^2 + X^4Z^2 + X^2Y^4 - 12X^2Y^2Z^2 + X^2Z^4 + 2Y^6 +Y^4Z^2 + Y^2Z^4 + 2Z^6)$$
$$=2^6(Y+Z)^2(X-Z)^2(X-Y)^2F(X,Y,Z)=0.$$

Also, one can show by direct computation that $$\alpha^5:[X:Y:Z]\mapsto [\overline{X}:\overline{Y}:\overline{Z}]$$
with $$\overline{X}=(Y - Z)^2 (Y + Z)^9 X    (X - Z)^{10} (X + Z)^5    (X - Y)^4   (X + Y),$$
$$\overline{Y}=(Y - Z)^2 (Y + Z)^9 Y    (X - Z)^{10} (X + Z)^5    (X - Y)^4   (X + Y),$$
$$\overline{Z}=(Y - Z)^2 (Y + Z)^9 Z    (X - Z)^{10} (X + Z)^5    (X- Y)^4   (X + Y),$$
implying $\alpha^5=1$. In this way, considering the group generated by the automorphisms defined up to now, we have that $120| |Aut(\mathcal{W})|$. 

The following lemma ensures that $Aut(\mathcal{W})$ is tame and hence from Theorem \ref{tame} that the Classical Hurwitz bound $|Aut(\mathcal{W})| \leq 84(g(\mathcal{W})-1)$ is satisfied. 

\begin{lemma} \label{autw1tame}
Let $\mathcal{W}$ be the Wiman's sextic defined as in Equation \eqref{eqWiman1} over a field of characteristic $p \geq 7$. Then $Aut(\mathcal{W})$ is tame. In particular $|Aut(\mathcal{W})| \leq 84(g(\mathcal{W})-1)=420$. Since also $120 \mid |Aut(\mathcal{W})|$ then $|Aut(\mathcal{W})| \in \{120,240,360\}$
\end{lemma}

\begin{proof}
Since $120$ divides $Aut(\mathcal{W})$, proving that $Aut(\mathcal{W})$ is tame also the second part of the claim follows from Theorem \ref{tame}.
\begin{itemize}
\item Assume that $p \geq 13$. If $Aut(\mathcal{W})>84(g-1)$, then from Theorem \ref{thmHurwitz}, $Aut(\mathcal{W})$ contains an auotomorphism $\sigma$ of order a multiple of $p$ and such that $\langle \sigma \rangle$ fixes some place $P$ of $\mathcal{W}$. This contradicts Theorem \ref{th11.78} as
$$\frac{(p-1)^2}{4} \geq 36=g^2.$$

\item Assume that $p=11$. Then by direct checking with MAGMA, $\mathcal{W}$ is an ordinary curve of genus $g=\gamma=6$. Suppose by contradiction that $Aut(\mathcal{W})$ is nontame and let $H$ be a Sylow $11$-subgroup of $Aut(\mathcal{W})$. From Theorem \ref{nak},
$$11 \leq |H| \leq \frac{11}{9}(6-1)<7,$$
a contradiction.

\item Hence we can assume that $p=7$. By direct checking with MAGMA, $\mathcal{W}$ is an ordinary curve of genus $g=\gamma=6$. Assuming again by contradiction that a Sylow $7$-subgroup $H$ of $Aut(\mathcal{W})$ is non-trivial, from Theorem \ref{nak} we have that
$$7 \leq |H| \leq \frac{7}{5}(6-1)=7.$$
Thus, $|H|=7$ and $\mathcal{W}$ is a Nakajima extremal curve. From Theorem \ref{nakext}, $|Aut(\mathcal{W})|$ divides $2p(p-1)=84$. Since $120$ divides $|Aut(\mathcal{W})|$ we have a contradiction.
\end{itemize}
\end{proof}

Consider the subgroup $G$ of $Aut(\mathcal{W})$ generated by the three involutions $\phi , \tau$ and $\rho$ as defined in Equation \ref{tauphirho}. In the following lemma the structure of $G$ is described. This forces $Aut(\mathcal{W})$ to contain a dihedral subgroup of order $8$. This combined with Lemma \ref{autw1tame} gives interesting constrains to the structure of $Aut(\mathcal{W})$.

\begin{lemma} \label{dihedral}
The subgroup $G$ of $Aut(\mathcal{W})$ with $G=\langle \tau,\phi,\rho \rangle$ is isomorphic to the dihedral group $D_8$ of order 8.
\end{lemma}
\begin{proof}
Since $\tau$ and $\phi$ commute, the group $H=\langle \tau, \phi \rangle$ is elementary abelian of order $4$. Also, $\rho$ normalizes $H$ as $\rho \phi \rho=\tau$. This shows that $|G|=8$, $G$ is not abelian and since $G$ contains at least four distincts involution ($\phi , \tau , \rho$ and $\phi \circ \tau $) it is necessarily isomorphic to the dihedral group $D_8$ of order $8$.
\end{proof}

From Lemma \ref{autw1tame}, if $p \geq 7$ there are just three possibilities for $|Aut(\mathcal{W})|$. In the following a case-by-case analysis is considered to obtain the main result of this section.
\begin{itemize}
\item \textbf{Case 1: $|Aut(\mathcal{W})| \in \lbrace 240, 360 \rbrace$}. Our aim is to prove that this case cannot occur. 
From Lemma \ref{mbt}, since $200=40(g-1)<|\aut(\mathcal{W})| < 84(g-1)$, we have that $Aut(\mathcal{W})=48(g-1)=240$, hence the case $|\aut(\mathcal{W})|=360$ cannot occur. 
Also, if $|\Aut(\mathcal{W})|=240$,  Lemma \ref{mbt} implies that $Aut(\mathcal{W})$ has a short orbit $O_3$ of length $|Aut(\mathcal{W}_1)|/8=30$.

 In particular, for $P \in O_3$, the stabilizer $Aut(\mathcal{W})_P$ has order $8$. From \cite[Lemma 11.44]{HKT} $Aut(\mathcal{W})_P$ is a cyclic group of order $8$. A Sylow $2$-subgroup of $Aut(\mathcal{W})$ has order $16$, contains a cyclic group of order $8$ and also a dihedral group of order $8$. By direct checking with MAGMA there are just $2$ groups of order $16$ containing both a cyclic group of order $8$ and a dihedral group of order $8$, namely $G \cong SmallGroup(16,i)$ with $i=7,8$. Again by direct checking with MAGMA, there are no groups of order $240$ whose Sylow $2$-subgroup is isomorphic to $SmallGroup(16,i)$ with $i=7,8$, and hence this case can be excluded.
This proves that $|Aut(\mathcal{W})|=120$.
\item \textbf{Case 2: $|Aut(\mathcal{W})|=120$}. From \cite[Theorem 11.79]{HKT} an abelian subgroup of $Aut(\mathcal{W})$ has order at most equal to $4g+4=28$. Using this information, one can check with MAGMA that the only groups of order $120$ with a structure which is compatible with the above condition are $SmallGroup(120,34)$, $SmallGroup(120,37)$ and $SmallGroup(120,38)$. 
The cases $Aut(\mathcal{W}) \cong SmallGroup(120,37)$ and $Aut(\mathcal{W}) \cong SmallGroup(120,37)$ are incompatible with Lemma \ref{gkaut} as they both have a minimal normal subgroup of order $4$ but a Sylow $2$-subgroup of $Aut(\mathcal{W})$ is isomorphic to the dihedral group $D_8$ which is  not a direct product of a cyclic group and a group of order 2. Finally we note that $SmallGroup(120,34)$ is isomorphic to  the symmetric group $S_5$.
\end{itemize}
The main result of this section is now proved combining the above lemmas with Remark \ref{ppiccoli}. It provides a positive characteristic analogue of a result of Wiman \cite{wiman} dealing with the structure of $Aut(\mathcal{W})$ over $\mathbb{C}$.

%but then the stabilizer of a point $P$ in $O_2$ should contain all the $2$-Sylow. 
%Thus it should contains $D_8$ but from $\cite{HKT}$, Lemma 11.44, a tame automorphism group which fix a %point must be cyclic.\\
%So it must be $|Aut(W)|=120$.\\
% Thus the only groups of order 120 compatible with $Aut(W)$ are $SmallGroup(120,34)$, $SmallGroup(120,37)$ %and $SmallGroup(120,38)$. \\
%The case $SmallGroup(120,37)$ and $SmallGroup(120,37)$ are incompatible with Lemma \ref{gkaut} as they %both have a minimal normal subgroup of order $4$ but a Sylow $2$-subgroup of $Aut(\mathcal{W}_1)$ is %isomorphic to the dihedral group $D_8$ which is  not a direct product of a cyclic group and a group of order 2.
 %Finally we note that $SmallGroup(120,34)$ is isomorphic to  the symmetric group $Sym_5$. We have proved then %the following Theorem.

\begin{theorem} \label{mainaut}
Let $\mathcal{W}$ denote the Wiman's sextic defined as in Equation \eqref{eqWiman1} over an algebraically closed field $\mathbb{K}$ of characteristic $p$. If $\mathbb{K}=\mathbb{C}$ or $p \geq 7$ then $Aut(\mathcal{W})$ is isomorphic to the symmetric group $S_5$. If $p=5$ then $\mathcal{W}$ is rational and $Aut(\mathcal{W}) \cong \PGL(2,\mathbb{K})$. If $p=2$ or $p=3$ then the homogeneous polynomial defining $\mathcal{W}$ as in \eqref{eqWiman1} is not irreducible.
\end{theorem}

\section{The $\mathbb{F}_{19^2}$-maximal Wiman Sextic $\mathcal{W}$ is  not Galois covered by the Hermitian curve $\cH_{19}$ over $\mathbb{F}_{19^2}$} \label{nonga}

 In this section we show that the $\mathbb{F}_{19^2}$-maximal Wiman's sextic $\mathcal{W}$ is not Galois covered by the Hermitian curve $\cH_{19}$ over $\mathbb{F}_{19^2}$. 

The proof relies on the results of \cite{MZ} and the main tools are the classification of automorphisms of the Hermitian curve based on their orders and geometrical properties, as stated in Lemma \ref{classificazione} and Theorem \ref{caratteri}, as well as the Riemann-Huwritz and Hilbert's formulas \eqref{RHformula} and \eqref{contributo}.
%\begin{theorem} {\rm \cite[Theorem 3.4.13]{Sti}} Riemann-Hurwitz formula:
%\begin{equation}\label{diff}
%2(g-1)=2|G|(g'-1)+\Delta  
%\end{equation}
%\end{theorem}
 %In the above formula we have that $ \Delta = \sum_{\sigma \in G \setminus \lbrace id \rbrace }i(\sigma)$, %where $i(\sigma) \geq 0$ is given by the Hilbert's different formula, see \cite{Sti} Thm. 3.8.7.\\

Assume by contradiction that $\mathcal{W}$ is Galois covered by $\cH_{19}$ over $\mathbb{F}_{19^2}$ and let $G \leq \PGU(3,19)$ denote the corresponding Galois group. Then
\begin{equation} \label{bounds}
\frac{|\cH_{19}|}{|\mathcal{W}(\mathbb{F}_{19^2})|} \leq |G| \leq \frac{2g(\cH_{19})-2}{2g(\mathcal{W})-2},
\end{equation}
see the proof of Theorem 5 in \cite{GK}.
Hence
$$11<\frac{19^3+1}{590} \leq |G| \leq \frac{2g(\cH_{19})-2}{10}=34.$$
%Since $\mathcal{W}_1$ has genus 6 and $\Delta \geq 0$, by \eqref{RHformula} we have that
% $$340 \geq 10|G|.$$
% Moreover $\cH_{19}$ has $19^3+1$ $\mathbb{F}_{19^2}$-rational points and $\mathcal{W}_1$ is %$\mathbb{F}_{19^2}$-maximal, hence $|\mathcal{W}_1(\mathbb{F}_{19^2})|=19^2+2\cdot 6\cdot %19+1=590$. It follows that $$590|G| \geq (19^3+1).$$
Moreover, since $G<\PGU(3,19)$, we have that $|G| \in \{12,\ldots,34\}$ divides $|\PGU(3,19)|$. This implies that  

\begin{equation} \label{ordini}
|G| \in \lbrace 12, 14,15,16,18,19,20,21,24,25,28,30,32 \rbrace.
\end{equation}

At this point a contradiction to $\mathcal{W}$ being Galois covered by $\cH_{19}$ over $\mathbb{F}_{19^2}$, is obtained with a case-by-case analysis with respect to $|G|$ as in \eqref{ordini}.
\begin{itemize}

\item \textbf{Case} $|G|=12$. By Sylow's Theorems $G$ contains at least two elements of order $3$. Then by Theorem \ref{caratteri}, $\Delta=i\cdot 2+k\cdot 20 $ with $i \geq 2$ and $k\leq 9$. This contradicts \eqref{RHformula}.

\item \textbf{Case} $|G|=14$. $G$ is either isomorphic to $C_{14}$ or to $D_{14}$. Since by Lemma \ref{classificazione} and Theorem \ref{caratteri}, $G$ cannot contain elements of order $14$, we have that $G \cong D_{14}$. Then by Theorem \ref{caratteri}, $\Delta=7\cdot 20+6\cdot 3=158$, contradicting \eqref{RHformula}.

\item \textbf{Case} $|G|=15$. By Sylow's Theorems $G$ is isomorphic to $C_{15}$. By Theorem \ref{caratteri}, $\Delta = 2\cdot 2+4\cdot i+8 \cdot 2$, with $i=0,20$ as a generator of the Sylow $5$-subgroup of $G$ is either of type (B1) or of type (A) in Lemma \ref{classificazione}. The Riemann-Hurwitz formula \eqref{RHformula} implies that  $\Delta=190$, a contradiction.

\item \textbf{Case} $|G|=18$. From Theorem \ref{caratteri} and Equation \eqref{RHformula}, $\Delta=20 \cdot i+2 \cdot (17-i)=160$, where $i$ is the number of involutions of $G$. So $i=7$. By Sylow's Theorems a group of order $18$ cannot contain exactly $7$ involutions, a contradiction.

\item \textbf{Case} $|G|=19$. Up to isomorphism $G$ is $C_{19}$. Since a power of an elation is an elation, by Theorem \ref{caratteri} we obtain that $\Delta=18i$ with $i=21$ or $i=2$. This contradict Equation \eqref{RHformula}.

\item \textbf{Case} $|G|=21$. By Sylow's Theorems and Schur-Zassenhaus Theorem \cite[Corollary 7.5]{machì}, $G$ is isomorphic either to $C_{21}$ or to the semidirect product $C_7 \rtimes C_3$. The former case is not possible because it contradicts Theorem \ref{caratteri}; the latter, again using Theorem \ref{caratteri}, contradicts 
Equation \eqref{RHformula}.

\item \textbf{Case} $|G|=25$. $G$ is isomorphic either to $C_{25}$ or to $C_5 \times C_5$. 
 The former case contradicts Theorem \ref{caratteri}.
 In the latter case, since every non-trivial element in $G$ has order $5$, $G$ contains only elements either of type (A) or of type (B1) from Lemma \ref{classificazione}. Since powers of a homology are homologies themselves, $\Delta$ must be divisible by $4$, contradicting Equation \eqref{RHformula}.

\item \textbf{Case} $|G|=28$. Since $G$ cannot contain elements of order $14$ from Theorem \ref{caratteri} and $G$ has exactly one Sylow $7$-subgroup, we have that $\Delta=i\cdot 20 + 6\cdot 3$ for some non-negative integer $i$. This contradicts Equation \eqref{RHformula}.

\item \textbf{Case} $|G|=30$. Since there are at least two elements of type (B3) in $G$ and $\Delta =40$ from Equation \eqref{RHformula}, we conclude that $G$ has a unique Sylow $2$-subgroup. This implies that $G \simeq C_{30}$ and, by Theorem \ref{caratteri}, $\Delta =1\cdot 20+2\cdot 3+4\cdot
 i+2\cdot 2+4\cdot k+8\cdot 2+8\cdot 2 \geq 62$ where $i$ is either equal to $20$ or to $0$ as the elements of order $5$ are either homologies or of type (B1), while $k$ is either equal to $20$ or to $0$ as the elements of order $10$ in $\PGU(3,19)$ are either homologies or of type (B1) from Lemma \ref{classificazione}. Combining the above arguments on $\Delta$ we deduce that this case cannot occur.
\end{itemize}

At this point, to conclude the proof of the main result of this section, we need to examine the following cases: $|G| \in \lbrace 16,20,24,32\rbrace$.
The arithmetical arguments used  for the previous cases are not sufficient, and hence normalizers of subgroups of $\PGU(3,19)$ need to be considered.

Assume that $\mathcal{W} \cong \cH_{19}/G$, and let $N$ be the normalizer $N$ of $G$ in $\PGU(3,19)$ and  $Q$ be the factor group $N/G$. From Galois theory $Q$ is a subgroup of $Aut(\mathcal{W}) \cong {\rm S}_5$. Since subgroups of the symmetric group $S_5$ can be completely listed, a contradiction can be obtained proving that the structure of $Q$ is not compatible with the subgroup structure of ${\rm S}_5$.

\begin{itemize}

\item \textbf{Case} $|G|=16$. We can use MAGMA to obtain the complete list, up to conjugation, of subgroups of $\PGU(3,19)$ of order $16$. Defining in MAGMA $S:=SubgroupLattice(\PGU(3,19): \ Properties:=true);$
we get that the subgroups $G$ of $\PGU(3,q)$ of order $16$ are: $G=SubgroupLattice(\PGU(3,19))[i]$ with $i \in \lbrace 70,71,72 \rbrace$. Considering the normalizer $N$ of $G$ in $\PGU(3,19)$ we get that $|Q|\in \lbrace 150,30,10 \rbrace$. Since $S_5$ has neither subgroups of order $150$ nor of order $30$, we get that the unique admissible case is $|Q|=10$. In this case $Q$ is cyclic but in $S_5$ each subgroup of order $10$ is dhiedral. We deduce that this case cannot occur.

\item \textbf{Case} $|G|=20$. As before, using MAGMA, we obtain the complete list of subgroups of order $20$ of $\PGU(3,19)$, namely $G=S[i]$ with $ i \in \lbrace 40, \dots ,51 \rbrace$. In these cases, $|Q| \in \lbrace 6840,40,20 \rbrace$. Clearly $S_5$ has no subgroups of orders $6840$ or $40$. Hence $|Q|=20$.  In this case $Q$ abelian. Since $S_5$ has no abelian subgroups of order $20$ we have a contradiction.

\item \textbf{Case} $|G|=24$. Using MAGMA, we obtain the complete list of subgroups of order $24$ in $\PGU(3,19)$: $G=S[i]$ with $ i \in \lbrace 73,\dots , 77 \rbrace$. In these cases, $|Q| \in \lbrace 30,20,10 \rbrace$. We note that $S_5$ contains no subgroups of orders $30$ while if $|Q|=20$ or $|Q|=10$ then $Q$ is cyclic. Since $S_5$ has no cyclic subgroups of order $20$ or $10$ we have a contradiction.

\item \textbf{Case} $|G|=32$. Arguing as before, we get that $G$ is isomorphic to $S[118]$. A contradiction is obtained combining Equation \eqref{RHformula} with the fact that $G$ contains seven involutions.
\end{itemize}

The main result of this section is now proved.
\begin{theorem} \label{galoisw1}
The $\mathbb{F}_{19}$-maximal Wiman's sextic $\mathcal{W}$ is not Galois covered by the Hermitian curve $\cH_{19}$ over $\mathbb{F}_{19^2}$.
\end{theorem}

\section*{Acknowledgments}

The authors would like to thank the Italian Ministry MIUR, Strutture Geometriche, Combinatoria e loro 
Applicazioni, Prin 2012 prot.~2012XZE22K and GNSAGA of the Italian INDAM.

This resarch was carried out within the project
``Progetto {\em Geometrie di Galois, Curve Algebriche su campi finiti e loro Applicazioni}'', supported by Fondo Ricerca di Base, 2015, of Universit\`a degli Studi di Perugia.

\vspace{1ex}
\noindent
Massimo Giulietti

\vspace{.5ex}
\noindent
Universit\'a degli Studi di Perugia,\\
Dipartimento di Matematica e Informatica,\\
Via Vanvitelli 1,\\
06123 Perugia,\\
Italy,\\
\emph{massimo.giulietti@unipg.it}

\vspace{1ex}
\noindent
Motoko Kawakita

\vspace{.5ex}
\noindent
Shiga University of Medical Science,\\
Seta Tsukinowa-cho\\
Otsu city, Shiga 520-2192\\
Japan,\\
\emph{kawakita@belle.shiga-med.ac.jp}

\vspace{1ex}
\noindent
Stefano Lia

\vspace{.5ex}
\noindent
%Universit\'a degli Studi di Perugia,\\
%Dipartimento di Matematica e Informatica,\\
%Via Vanvitelli 1,\\
%06123 Perugia,\\
%Italy,\\
\emph{stefano.lia.7@gmail.com}

\vspace{1ex}
\noindent
Maria Montanucci

\vspace{.5ex}
\noindent
Universit\'a degli Studi della Basilicata,\\
Dipartimento di Matematica, Informatica ed Economia,\\
Campus di Macchia Romana,\\
Viale dell' Ateneo Lucano 10,\\
85100 Potenza,\\
Italy,\\
\emph{maria.montanucci@unibas.it}

\end{document}